\documentclass[10pt]{article}

\usepackage[hmargin=35mm,vmargin=40mm]{geometry}

\usepackage{amsmath}              
\usepackage{amssymb}              
\usepackage{amsthm}               
\usepackage{enumerate}            
\usepackage{mathrsfs}             
\usepackage{graphicx}             
\usepackage[tight]{subfigure}     

\newtheorem{theorem}{Theorem}
\newtheorem{lemma}[theorem]{Lemma}
\newtheorem{corollary}[theorem]{Corollary}
\newtheorem{algorithm}[theorem]{Algorithm}
\newtheorem{conjecture}{Conjecture} 

\theoremstyle{definition}
\newtheorem*{defn}{Definition}
\newtheorem*{remark}{Remark}

\newcounter{ctr-numvert}
\newcounter{ctr-can-fast}

\newcommand{\pg}[1]{\mathscr{P}(#1)}
\newcommand{\rpg}[1]{\mathscr{P}_1(#1)}

\newcommand{\regina}{\emph{Regina}}
\newcommand{\sig}{\sigma}
\newcommand{\tri}{\mathcal{T}}

\begin{document}

\title{The Pachner graph and the simplification \\ of 3-sphere triangulations}
\author{Benjamin A.~Burton}
\date{February 23, 2011} 

\maketitle

\begin{abstract}
    It is important to have fast and effective methods for simplifying
    3-manifold triangulations without losing any topological information.
    In theory this is difficult: we might need to make a triangulation
    super-exponentially more complex before we can make it smaller than its
    original size.  Here we present experimental work suggesting that for
    3-sphere triangulations the reality is far different: we never need to
    add more than two tetrahedra, and we never need more than a handful
    of local modifications.  If true in general, these extremely
    surprising results
    would have significant implications for decision algorithms and the
    study of triangulations in 3-manifold topology.

    The algorithms behind these experiments are interesting
    in their own right.  Key techniques include
    the isomorph-free generation of all 3-manifold triangulations of a
    given size, polynomial-time computable
    signatures that identify triangulations
    uniquely up to isomorphism, and parallel algorithms for studying
    finite level sets in the infinite Pachner graph.

    \medskip
    \noindent \textbf{ACM classes}\quad
    F.2.2; G.2.1; G.2.2; D.1.3

    \medskip
    \noindent \textbf{Keywords}\quad
    Triangulations, 3-manifolds, Pachner moves, 3-sphere recognition,
    isomorph-free enumeration
\end{abstract}

%
%

\section{Introduction}

Triangulations of 3-manifolds are ubiquitous in
computational knot theory
and low-dimensional topology.  They are easily obtained and offer
a natural setting for many important algorithms.

Computational topologists typically allow triangulations in which the
constituent tetrahedra may be ``bent'' or ``twisted'', and where
distinct edges or vertices of the same tetrahedron
may even be joined together.  Such triangulations (sometimes called
\emph{semi-simplicial} or \emph{pseudo-triangulations})
can describe rich topological structures using remarkably few tetrahedra.
For example, the 3-dimensional sphere can be built from
just one tetrahedron, and more complex spaces such as non-trivial
surface bundles can be built from as few as six \cite{matveev90-complexity}.

An important class of triangulations is the
\emph{one-vertex triangulations}, in which all vertices of all tetrahedra
are joined together as a single point.  These are simple to obtain
\cite{jaco03-0-efficiency,matveev03-algms},
and they are often easier to deal with both theoretically and computationally
\cite{burton10-dd,jaco02-algorithms-essential,matveev03-algms}.

Keeping the number of tetrahedra small is crucial in computational
topology, since many important algorithms are exponential (or even
super-exponential) in the number of tetrahedra
\cite{burton10-complexity,burton10-dd}.
To this end, topologists have developed a rich suite of local
simplification moves that allow us to reduce the number of
tetrahedra without losing any topological information
\cite{burton04-facegraphs,matveev98-recognition}.

The most basic of these are the four \emph{Pachner moves}
(also known as \emph{bistellar moves}).  These include the
3-2 move (which reduces the number of tetrahedra but preserves the
number of vertices), the 4-1 move (which reduces both numbers), and also
their inverses, the 2-3 and 1-4 moves.
It is known that any two triangulations of the same closed
3-manifold are related by a sequence of Pachner moves
\cite{pachner91-moves}.  Moreover, if both are one-vertex
triangulations then we can relate them using 2-3 and 3-2 moves alone
\cite{matveev03-algms}.

However, little is known about how \emph{difficult} it is to relate two
triangulations by a sequence of Pachner moves.  In a series of papers,
Mijatovi{\'c} develops upper bounds on the number of moves required for
various classes of 3-manifolds \cite{mijatovic03-simplifying,mijatovic04-sfs,
mijatovic05-knot,mijatovic05-haken}.
All of these bounds are super-exponential in the number of tetrahedra,
and some even involve exponential towers of exponential functions.
For relating one-vertex triangulations using only 2-3 and 3-2 moves,
no explicit bounds are known at all.

In this paper we focus on one-vertex triangulations of
the 3-sphere.  Here simplification is tightly linked to the important
problem of \emph{3-sphere recognition}, where we are given an input
triangulation $\tri$ and asked whether $\tri$ represents the 3-sphere.
This problem plays an key role in other important topological
algorithms, such as connected sum decomposition
\cite{jaco03-0-efficiency,jaco95-algorithms-decomposition} and
unknot recognition \cite{hara05-unknotting},
and it is now becoming important in computational \emph{4-manifold} topology.
We can use Pachner moves for 3-sphere recognition in two ways:
\begin{itemize}
    \item They give us a \emph{direct} 3-sphere recognition
    algorithm: try all possible sequences of
    Pachner moves on $\tri$ up to Mijatovi{\'c}'s upper bound,
    and return ``true'' if and only if we reach one of the well-known
    ``canonical'' 3-sphere triangulations with one or two tetrahedra.

    \item They also allow a \emph{hybrid}
    recognition algorithm: begin with a fast and/or greedy
    procedure to simplify $\tri$ as far as possible within a limited
    number of moves.  If we reach a canonical 3-sphere triangulation
    then return ``true''; otherwise run a more traditional 3-sphere
    recognition algorithm on our new (and hopefully simpler) triangulation.
\end{itemize}

The direct algorithm lies well outside the realm of feasibility:
Mijatovi{\'c}'s bound is super-exponential in the number of tetrahedra,
and the running time is at least exponential in Mijatovi{\'c}'s bound.
Current implementations \cite{burton04-regina} use the
hybrid method, which is extremely effective in practice.
Experience suggests that when $\tri$ \emph{is} the 3-sphere, the greedy
simplification almost always gives a canonical triangulation.
If simplification fails, we revert to the traditional algorithm of
Rubinstein \cite{rubinstein95-3sphere}; although this runs in exponential
time, recent improvements by several authors have made it
extremely effective for moderate-sized problems
\cite{burton10-dd,burton10-quadoct,jaco03-0-efficiency,
thompson94-thinposition}.\footnote{%
    See \cite{burton10-quadoct} for explicit measurements of running time.}

Our aims in this paper are:
\begin{itemize}
    \item to measure how easy or difficult it is \emph{in practice}
    to relate two triangulations of the 3-sphere using Pachner moves, or
    to simplify a 3-sphere triangulation to use fewer tetrahedra;

    \item to understand why greedy simplification techniques work so
    well in practice, despite the prohibitive theoretical bounds of
    Mijatovi{\'c};

    \item to investigate the possibility that Pachner moves could be
    used as the basis for a direct 3-sphere recognition algorithm
    that runs in sub-exponential or even polynomial time.
\end{itemize}

Fundamentally this is an experimental paper (though the theoretical
underpinnings are interesting in their own right).
Based on an exhaustive study of $\sim 150$~million
triangulations (including $\sim 31$~million one-vertex triangulations of the
3-sphere), the answers appear to be:
\begin{itemize}
    \item we can relate and simplify one-vertex triangulations of
    the 3-sphere using remarkably few Pachner moves;

    \item both procedures require us to add \emph{at most two}
    extra tetrahedra, which explains why greedy simplification works so well;

    \item the number of moves required to simplify such a
    triangulation could also be bounded by a constant,
    which means polynomial-time 3-sphere recognition
    may indeed be possible.
\end{itemize}

These observations are extremely surprising, especially in light of
Mijatovi{\'c}'s bounds.  If they can be proven in general---yielding
a polynomial-time 3-sphere recognition algorithm---this would be a
significant breakthrough in computational topology.

In Section~\ref{s-prelim} we outline preliminary concepts and
introduce the \emph{Pachner graph}, an infinite graph whose nodes
represent triangulations and whose arcs represent Pachner moves.
This graph is the framework on which we build the rest of the paper.
We define \emph{simplification paths} through the graph,
as well as the key quantities of \emph{length} and \emph{excess height}
that we seek to measure.

We follow in Section~\ref{s-tools} with two key tools for studying the
Pachner graph: an isomorph-free census of all closed 3-manifold
triangulations with $\leq 9$ tetrahedra (which gives us the nodes of the
graph), and \emph{isomorphism signatures} of triangulations that can be
computed in polynomial time (which allow us to construct the arcs of the
graph).

Section~\ref{s-analysis} describes parallel algorithms for bounding
both the length and excess height of simplification paths, and
presents the highly unexpected experimental results outlined above.
We finish in Section~\ref{s-conc} with a discussion of the implications
and consequences of these results.

%
%

\section{Triangulations and the Pachner graph} \label{s-prelim}

A \emph{3-manifold triangulation of size $n$} is a collection of
$n$ tetrahedra whose $4n$ faces are affinely identified
(or ``glued together'') in $2n$ pairs so that the resulting
topological space is a closed 3-mani\-fold.\footnote{%
    It is sometimes useful to consider \emph{bounded} triangulations
    where some faces are left unidentified, or \emph{ideal}
    triangulations where the overall space only becomes a 3-manifold
    when we delete the vertices of each tetrahedron.
    Such triangulations do not concern us here.}
We are not interested in the shapes or sizes of tetrahedra (since these
do not affect the topology), but merely the combinatorics of how the faces are
glued together.
Throughout this paper, all triangulations and 3-manifolds are assumed to
be connected.

We do allow two faces of the same tetrahedron to be identified, and we also
note that distinct edges or vertices of the same tetrahedron might
become identified as a by-product of the face gluings.
A set of tetrahedron vertices that are identified together is collectively
referred to as a \emph{vertex of the triangulation}; we define an
\emph{edge} or \emph{face of the triangulation}
in a similar fashion.

\begin{figure}[htb]
    \centering
    \includegraphics{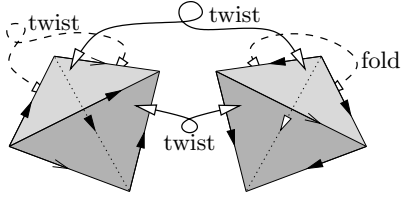}
    \caption{A 3-manifold triangulation of size $n=2$}
    \label{fig-rp3}
\end{figure}

Figure~\ref{fig-rp3} illustrates a 3-manifold triangulation
of size $n=2$.
Here the back two faces of the first tetrahedron are identified with a
twist, the front faces of the first tetrahedron are identified with the
front faces of the second using more twists, and the back faces of the
second tetrahedron are identified together by directly ``folding'' one
onto the other.
This is a \emph{one-vertex triangulation} since all eight tetrahedron
vertices become identified together.  The triangulation has three distinct
edges, indicated in the diagram by three distinct arrowheads.

Mijatovi{\'c} \cite{mijatovic03-simplifying}
describes a \emph{canonical triangulation} of the 3-sphere of size $n=2$,
formed by a direct identification of the boundaries of two tetrahedra.
In other words, given two tetrahedra $\mathit{ABCD}$ and $A'B'C'D'$,
we directly identify face $\mathit{ABC}$ with $A'B'C'$, $ABD$ with $A'B'D'$,
and so on.  The resulting triangulation has four faces,
six edges, and four vertices.

The four \emph{Pachner moves} describe local
modifications to a triangulation.  These include:
\begin{itemize}
    \item the \emph{2-3 move}, where we replace two distinct
    tetrahedra joined along a common face with three distinct tetrahedra
    joined along a common edge;
    \item the \emph{1-4 move}, where we replace a single
    tetrahedron with four distinct tetrahedra meeting at a common internal
    vertex;
    \item the \emph{3-2} and \emph{4-1 moves}, which are inverse to the
    2-3 and 1-4 moves.
\end{itemize}

\begin{figure}[htb]
    \centering
    \subfigure[The 2-3 and 3-2 moves]{%
        \label{sub-pachner-23} \includegraphics[scale=0.45]{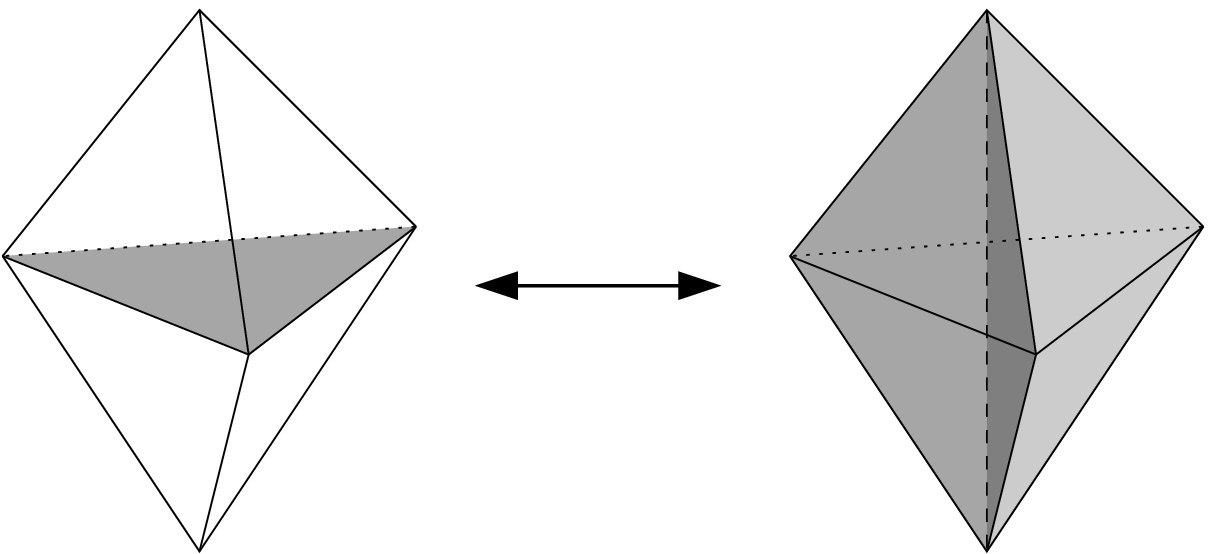}}
    \hspace{1.5cm}
    \subfigure[The 1-4 and 4-1 moves]{%
        \label{sub-pachner-14} \includegraphics[scale=0.45]{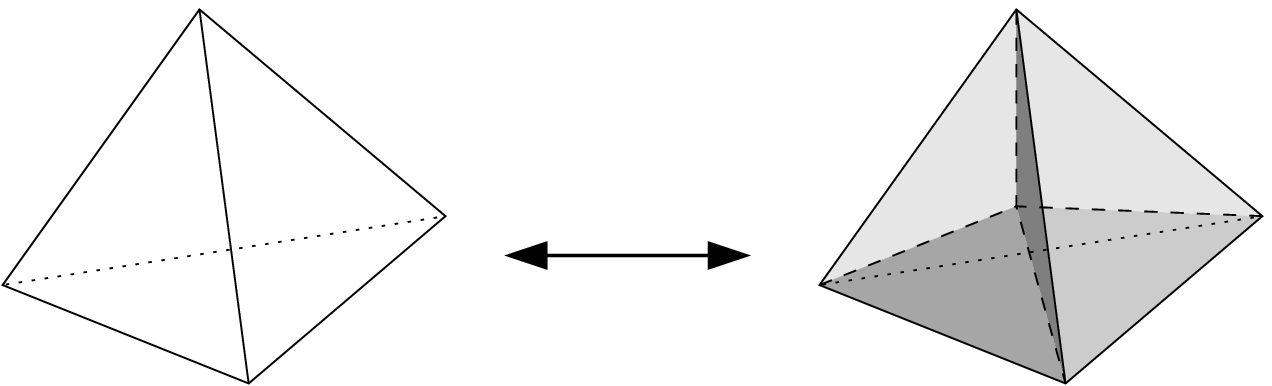}}
    \caption{The four Pachner moves for a 3-manifold triangulation}
    \label{fig-pachner}
\end{figure}

These four moves are illustrated in Figure~\ref{fig-pachner}.
Essentially, the 1-4 and 4-1 moves retriangulate the interior of a
pyramid, and the 2-3 and 3-2 moves retriangulate the interior of a
bipyramid.
It is clear that Pachner moves do not change the topology of the
triangulation (i.e., the underlying 3-manifold remains the same).
Another important observation is that the 2-3 and 3-2 moves do not change
the number of vertices in the triangulation.

Two triangulations are \emph{isomorphic} if they are identical up to a
relabelling of tetrahedra and a reordering of the four vertices of each
tetrahedron (that is, isomorphic in the usual combinatorial sense).
Up to isomorphism, there are finitely many distinct triangulations of
any given size.

Pachner originally showed that any two
triangulations of the same closed 3-manifold can be made isomorphic by
performing a sequence of Pachner moves \cite{pachner91-moves}.\footnote{%
    As Mijatovi{\'c} notes,
    Pachner's original result was proven only for true simplicial complexes,
    but it is easily extended to the more flexible definition of a
    triangulation that we use here \cite{mijatovic03-simplifying}.
    The key step is to remove irregularities by performing
    a second barycentric subdivision using Pachner moves.}
Matveev later strengthened this result to show that any two
\emph{one-vertex} triangulations of the same closed 3-manifold
with at least two tetrahedra can be made isomorphic through a sequence
of 2-3 and/or 3-2 moves \cite{matveev03-algms}.
The two-tetrahedron condition is required because it is impossible to
perform a 2-3 or 3-2 move upon a one-tetrahedron triangulation (each move
requires two or three distinct tetrahedra).

In this paper we introduce the \emph{Pachner graph}, which describes
\emph{how} distinct triangulations of a closed 3-manifold can be
related via Pachner moves.
We define this graph in terms of \emph{nodes} and \emph{arcs}, to avoid
confusion with the \emph{vertices} and \emph{edges} that appear in
3-manifold triangulations.

\begin{defn}[Pachner graph]
    Let $M$ be any closed 3-manifold.  The \emph{Pachner graph} of $M$,
    denoted $\pg{M}$, is an infinite graph constructed as follows.
    The nodes of $\pg{M}$ correspond to isomorphism classes of
    triangulations of $M$.  Two nodes of $\pg{M}$ are joined by an
    arc if and only if there is some Pachner move that converts
    one class of triangulations into the other.

    The \emph{restricted Pachner graph} of $M$, denoted $\rpg{M}$,
    is the subgraph of $\pg{M}$ defined by only those nodes
    corresponding to one-vertex triangulations.
    The nodes of $\pg{M}$ and $\rpg{M}$ are partitioned into finite
    \emph{levels} $1,2,3,\ldots$, where each level~$n$ contains the nodes
    corresponding to $n$-tetrahedron triangulations.
\end{defn}

\begin{figure}[htb]
    \centering
    \includegraphics{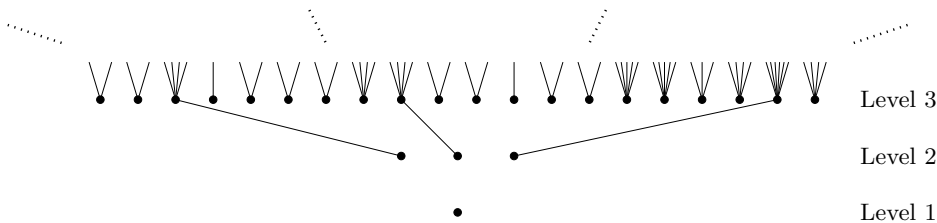}
    \caption{Levels 1--3 of the restricted Pachner graph of the 3-sphere}
    \label{fig-rpg-s3}
\end{figure}

It is clear that the arcs are well-defined (since Pachner moves are
preserved under isomorphism), and that arcs do not need to be
directed (since each 2-3 or 1-4 move has a corresponding inverse
3-2 or 4-1 move).  In the full Pachner graph $\pg{M}$, each arc
runs from some level $i$ to a nearby level $i\pm1$ or $i\pm3$.
In the restricted Pachner graph
$\rpg{M}$, each arc must describe a 2-3 or 3-2 move, and must run from some
level $i$ to an adjacent level $i\pm1$.
Figure~\ref{fig-rpg-s3} shows the first few levels of the restricted Pachner
graph of the 3-sphere.

We can now reformulate the results of Pachner and Matveev as follows:

\begin{theorem}[Pachner, Matveev] \label{t-connected}
    The Pachner graph of any closed 3-manifold is connected.
    If we delete level~1,
    the restricted Pachner graph of any closed 3-manifold is also connected.
\end{theorem}

To simplify a triangulation we essentially follow a path through
$\pg{M}$ or $\rpg{M}$ from a higher level to a lower level,
which motivates the following definition.

\begin{defn}[Simplification path]
    A \emph{simplification path} is a directed path through either
    $\pg{M}$ or $\rpg{M}$ from a node at some level $i$ to
    a node at some lower level $<i$.
    The \emph{length} of a simplification path is the number of arcs it
    contains.  The \emph{excess height} of a simplification path
    is the smallest $h \geq 0$ for which
    the entire path stays in or below level $i+h$.
\end{defn}

Both the length and excess height measure how difficult it is to simplify a
triangulation: the length measures the number of Pachner moves,
and the excess height measures the number of extra tetrahedra required.
For the 3-sphere, the only known bounds on these quantities
are the following:

\begin{theorem}[Mijatovi{\'c} \cite{mijatovic03-simplifying}] \label{t-mij}
    Any triangulation of the 3-sphere can be converted into the
    canonical triangulation using less than
    $6 \cdot 10^6 n^2 2^{2 \cdot 10^4 n^2}$ Pachner moves.
\end{theorem}

\begin{corollary}
    In the Pachner graph of the 3-sphere, from any node at level $n>2$
    there is a simplification path of length
    less than $6 \cdot 10^6 n^2 2^{2 \cdot 10^4 n^2}$ and excess height
    less than $3 \cdot 10^6 n^2 2^{2 \cdot 10^4 n^2}$.
\end{corollary}

In the \emph{restricted} Pachner graph, no explicit bounds on these
quantities are known at all.

%
%

\section{Key tools} \label{s-tools}

Experimental studies of the Pachner graph are difficult: the graph
itself is infinite, and even the finite level sets grow
super-exponentially in size.  By working with isomorphism classes of
triangulations, we keep the level sets considerably smaller than if
we had used labelled triangulations instead.
However, the trade-off is that both the nodes and the arcs of the
graph are more difficult to construct.

In this section we outline two key algorithmic tools for studying the
Pachner graph: a \emph{census of triangulations} (which enumerates the
nodes at each level), and polynomial-time computable
\emph{isomorphism signatures} (which allow us to construct the arcs).

\subsection{A census of triangulations} \label{s-tools-census}

To enumerate the nodes of Pachner graphs, we build a census of all
3-manifold triangulations of size $n \leq 9$, with each triangulation
included precisely once up to isomorphism.  Because we are particularly
interested in one-vertex triangulations as well as triangulations of the
3-sphere, we extract such
triangulations into separate censuses with the help of the highly optimised
3-sphere recognition algorithm described in \cite{burton10-quadoct}.
The final counts are summarised in Table~\ref{tab-census}.

\begin{table}[htb]
\centering
\small
\begin{tabular}{c|r|r|r|r}
Number of & \multicolumn{2}{c|}{All closed 3-manifolds} &
\multicolumn{2}{c}{3-spheres only} \\
tetrahedra &
\multicolumn{1}{c|}{No constraints} &
\multicolumn{1}{c|}{1-vertex only} &
\multicolumn{1}{c|}{No constraints} &
\multicolumn{1}{c}{1-vertex only} \\
\hline
1   &                4 &                3 &             2 &             1 \\
2   &               17 &               12 &             6 &             3 \\
3   &               81 &               63 &            32 &            20 \\
4   &              577 &              433 &           198 &           128 \\
5   &           5\,184 &           3\,961 &        1\,903 &        1\,297 \\
6   &          57\,753 &          43\,584 &       19\,935 &       13\,660 \\
7   &         722\,765 &         538\,409 &      247\,644 &      169\,077 \\
8   &      9\,787\,509 &      7\,148\,483 &   3\,185\,275 &   2\,142\,197 \\
9   &    139\,103\,032 &     99\,450\,500 &  43\,461\,431 &  28\,691\,150 \\
\hline
Total & 149\,676\,922 & 107\,185\,448 & 46\,916\,426 & 31\,017\,533
\end{tabular}
\caption{Counts for 3-manifold triangulations of various types in the census}
\label{tab-census}
\end{table}

The algorithms behind this census are sophisticated; see
\cite{burton07-nor10} for some of the techniques involved.
The constraint that the triangulation must represent a 3-manifold is
critical: if we just enumerate all
pairwise identifications of faces up to isomorphism, there are at least
\[ \frac{[(4n-1)\times(4n-3)\times\cdots\times3\times1]\cdot6^{2n}}
    {n! \cdot 24^n} \quad \simeq \quad 2.35 \times 10^{16} \]
possibilities for $n=9$.  To enforce the 3-manifold constraint we use a
modified union-find algorithm that tracks
partially-constructed edge links and vertex links; see
\cite{burton07-nor10} for details.

Even with this constraint, we can prove that the census grows at a
super-exponential rate:

\setcounter{ctr-numvert}{\arabic{theorem}}
\begin{theorem} \label{t-numvert}
    The number of distinct isomorphism classes of 3-manifold
    triangulations of size $n$ grows at an asymptotic rate of
    $\exp(\Theta(n\log n))$.
\end{theorem}

The proof is detailed, and is given in the appendix.

For the largest case $n=9$, the enumeration of all 3-manifold
triangulations up to isomorphism
required $\sim 85$ days of CPU time as measured on a
single 1.7~GHz IBM Power5 processor (though in reality this was reduced to
2--3 days of wall time using 32 CPUs in parallel).
The time required to extract all 3-sphere triangulations
from this census was negligible in comparison.

\subsection{Isomorphism signatures}

To construct arcs of the Pachner graph, we begin at a node---that is,
a 3-manifold triangulation $\tri$---and perform Pachner moves.  Each Pachner
move results in a new triangulation $\tri'$, and our main difficulty
is in deciding which node of the Pachner graph represents $\tri'$.

A na\"ive approach might be to search through nodes at the appropriate
level of the Pachner graph and test each corresponding triangulation
for isomorphism with $\tri'$.  However, this approach is infeasible:
although isomorphism testing is fast (as we prove below), the sheer number
of nodes at level $n$ of the graph is too large
(see Theorem~\ref{t-numvert}).

What we need is a property of the triangulation $\tri'$ that is easy to
compute, and that uniquely defines the isomorphism class of $\tri'$.
This property could be used as the key in a data structure with fast
insertion and fast lookup (such as a hash table or a red-black tree),
and by computing this property we could quickly jump to
the relevant node of the Pachner graph.

Here we define such a property, which we call the \emph{isomorphism
signature} of a triangulation.
In Theorem~\ref{t-sig-unique} we show that isomorphism signatures do
indeed uniquely define isomorphism classes, and in
Theorem~\ref{t-sig-fast} we show that they are small to store
and fast to compute.

A \emph{labelling} of a triangulation of size $n$ involves:
(i)~numbering its tetrahedra from 1 to $n$ inclusive, and
(ii)~numbering the four vertices of each tetrahedron from 1 to 4
inclusive.  We also label the four faces of each tetrahedron from 1 to 4
inclusive so that face $i$ is opposite vertex $i$.
A key ingredient of isomorphism signatures is
\emph{canonical labellings}, which we define as follows.

\begin{defn}[Canonical labelling]
    Given a labelling of a triangulation of size $n$,
    let $A_{t,f}$ denote the tetrahedron which is glued to face
    $f$ of tetrahedron $t$ (so that $A_{t,f} \in \{1,\ldots,n\}$ for
    all $t=1,\ldots,n$ and $f=1,\ldots,4$).  The labelling is
    \emph{canonical} if, when we write out the sequence
    $A_{1,1},A_{1,2},A_{1,3},A_{1,4},\allowbreak A_{2,1},\ldots,A_{n,4}$,
    the following properties hold:
    \begin{enumerate}[(i)]
        \item For each $2 \leq i < j$,
        tetrahedron $i$ first appears before tetrahedron $j$
        first appears.
        \item For each $i \geq 2$, suppose tetrahedron $i$ first appears
        as the entry $A_{t,f}$.  Then the corresponding gluing
        uses the \emph{identity map}:
        face $f$ of tetrahedron $t$ is glued to face $f$ of tetrahedron $i$
        so that vertex $v$ of tetrahedron $t$ maps to vertex $v$ of
        tetrahedron $i$ for each $v \neq f$.
    \end{enumerate}
\end{defn}

As an example, consider the triangulation of size $n=3$ described by
Table~\ref{tab-gluings}.  This table lists the precise gluings of
tetrahedron faces.  For instance, the second cell in the bottom row
indicates that face~2 of tetrahedron~3 is glued to tetrahedron~2, in
such a way that
vertices $1,3,4$ of tetrahedron~3 map to vertices $4,2,3$ of tetrahedron~2
respectively.
This same gluing can be seen from the other direction by examining
the first cell in the middle row.

\begin{table}[htb]
    \newcommand{\gap}{\hspace{2ex}}
    \centering
    \small
    \begin{tabular}{l|c|c|c|c}
    &
    \multicolumn{1}{c|}{Face 1} &
    \multicolumn{1}{c|}{Face 2} &
    \multicolumn{1}{c|}{Face 3} &
    \multicolumn{1}{c}{Face 4} \\
    & Vertices 234 & Vertices 134 & Vertices 124 & Vertices 123 \\
    \hline
    {Tet.\ 1} & Tet.\ 1:\gap231 & Tet.\ 2:\gap134 &
                Tet.\ 3:\gap124 & Tet.\ 1:\gap423 \\
    {Tet.\ 2} & \framebox{Tet.\ 3:\gap341} & Tet.\ 1:\gap134 &
                Tet.\ 2:\gap123 & Tet.\ 2:\gap124 \\
    {Tet.\ 3} & Tet.\ 3:\gap123 & \framebox{Tet.\ 2:\gap423} &
                Tet.\ 1:\gap124 & Tet.\ 3:\gap234
    \end{tabular}
    \caption{The tetrahedron face gluings for an example 3-tetrahedron
        triangulation}
    \label{tab-gluings}
\end{table}

It is simple to see that the labelling for this triangulation is canonical.
The sequence $A_{1,1},\ldots,A_{n,4}$ is
$1,2,3,1,\allowbreak 3,1,2,2,\allowbreak 3,2,1,3$
(reading tetrahedron numbers from left to right and then top to bottom
in the table), and tetrahedron~2 first appears before
tetrahedron~3 as required.  Looking closer,
the first appearance of tetrahedron~2 is in the second cell of the top
row where vertices $1,3,4$ map to $1,3,4$, and the first appearance of
tetrahedron~3 is in the following cell where vertices $1,2,4$ map to
$1,2,4$.  In both cases the gluings use the identity map.

\setcounter{ctr-can-fast}{\arabic{theorem}}
\begin{lemma} \label{l-can-fast}
    For any triangulation $\tri$ of size $n$, there are precisely
    $24n$ canonical labellings of $\tri$, and these can be enumerated in
    $O(n^2\log n)$ time.
\end{lemma}

\begin{proof}
    In summary, we can choose any of the $n$ tetrahedra to label as
    tetrahedron~1, and we can choose any of the $4!=24$ labellings of its
    four vertices.  From here the remaining labels
    are forced, and can be deduced in $O(n\log n)$ time.
    The full proof is given in the appendix.
\end{proof}

\begin{defn}[Isomorphism signature]
    For any triangulation $\tri$ of size $n$,
    enumerate all $24n$ canonical labellings of $\tri$,
    and for each canonical labelling encode the full set of
    face gluings as a sequence of bits.
    We define the \emph{isomorphism signature} to be the
    lexicographically smallest of these $24n$ bit sequences, and we
    denote this by $\sig(\tri)$.
\end{defn}

To encode the full set of face gluings for a canonical labelling, we
could simply convert a table of gluing data (such as Table~\ref{tab-gluings})
into a sequence of bits.  For practical implementations we use a more
compact representation, which will be described
in the full version of this paper.

\begin{theorem} \label{t-sig-unique}
    Given two 3-manifold triangulations $\tri$ and $\tri'$,
    we have $\sig(\tri) = \sig(\tri')$ if and only if
    $\tri$ and $\tri'$ are isomorphic.
\end{theorem}

\begin{proof}
    It is clear that $\sig(\tri) = \sig(\tri')$ implies that
    $\tri$ and $\tri'$ are isomorphic, since both signatures encode the
    same gluing data.  Conversely, if $\tri$ and $\tri'$ are isomorphic
    then their $24n$ canonical labellings are the same (though they
    might be enumerated in a different order).  In particular, the
    lexicographically smallest canonical labellings will be identical;
    that is, $\sig(\tri)=\sig(\tri')$.
\end{proof}

\begin{theorem} \label{t-sig-fast}
    Given a 3-manifold triangulation $\tri$ of size $n$,
    the isomorphism signature $\sig(\tri)$ has $O(n\log n)$ size
    and can be generated in $O(n^2\log n)$ time.
\end{theorem}

\begin{proof}
    To encode a full set of face gluings, at worst we require a table
    of gluing data such as Table~\ref{tab-gluings}, with $4n$ cells each
    containing four integers.
    Because some of these integers require $O(\log n)$ bits
    (the tetrahedron labels), it follows that the total size of
    $\sig(\tri)$ is $O(n \log n)$.

    The algorithm to generate $\sig(\tri)$ is spelled out
    explicitly in its definition.  The $24n$ canonical labellings of
    $\tri$ can be enumerated in $O(n^2\log n)$ time (Lemma~\ref{l-can-fast}).
    Because a full set of face gluings has size $O(n\log n)$,
    we can encode the $24n$ bit sequences and select the
    lexicographically smallest in $O(n^2\log n)$ time,
    giving a time complexity of $O(n^2\log n)$ overall.
\end{proof}

This space complexity of $O(n\log n)$ is the best we can hope
for, since Theorem~\ref{t-numvert} shows that the number of distinct
isomorphism signatures for size $n$ triangulations grows like
$\exp(\Theta(n \log n))$.

It follows from Theorems~\ref{t-sig-unique} and~\ref{t-sig-fast}
that isomorphism signatures are ideal tools for constructing arcs in the
Pachner graph, as explained at the beginning of this section.
Moreover, the relevant definitions and results are easily extended
to bounded and ideal triangulations (which are beyond the scope of this paper).
We finish with a simple but important consequence of our results:

\begin{corollary}
    Given two 3-manifold triangulations $\tri$ and $\tri'$ each of size $n$,
    we can test whether $\tri$ and $\tri'$ are isomorphic in
    $O(n^2\log n)$ time.
\end{corollary}

%
%

\section{Analysing the Pachner graph} \label{s-analysis}

As discussed in the introduction, our focus
is on one-vertex triangulations of the 3-sphere.
We therefore direct our attention to $\rpg{S^3}$, the restricted Pachner
graph of the 3-sphere.

In this section we develop algorithms to bound the shortest length and
smallest excess height of any simplification path from a given node at level
$n$ of $\rpg{S^3}$.  By running these algorithms over the full census of
$31\,017\,533$ one-vertex triangulations of the 3-sphere (as described
in Section~\ref{s-tools-census}), we obtain a
computer proof of the following results:

\begin{theorem} \label{t-results}
    From any node at level $n$ of the graph $\rpg{S^3}$
    where $3 \leq n \leq 9$, there is a simplification path of length
    $\leq 13$, and there is a simplification path of excess height $\leq 2$.
\end{theorem}

The bound $3 \leq n$ is required because there are no simplification
paths in $\rpg{S^3}$ starting at level~2 or below (see
Figure~\ref{fig-rpg-s3}).  For $n > 9$
a computer proof becomes computationally infeasible.

The results of Theorem~\ref{t-results} are astonishing, especially in light
of Mijatovi{\'c}'s super-exponential bounds.  Furthermore, whilst it can
be shown that the excess height bound of $\leq 2$ is tight,
the length estimate of
$\leq 13$ is extremely rough: the precise figures could be much
smaller still.  These results have important implications, which we
discuss later in Section~\ref{s-conc}.

In this section we describe the algorithms behind
Theorem~\ref{t-results}, and we present the experimental results
in more detail.  Our algorithms are constrained by the following factors:
\begin{itemize}
    \item Their time and space complexities must be
    close to linear in the number of nodes that they examine,
    due to the sheer size of the census.

    \item They cannot loop through all nodes in $\rpg{S^3}$, since
    the graph is infinite.  They cannot even loop through all nodes
    at level $n \geq 10$, since there are too many to enumerate.

    \item They cannot follow arbitrary breadth-first or depth-first
    searches through $\rpg{S^3}$, since the graph is infinite and
    can branch heavily in the upward direction.\footnote{%
        In general, a node at level $n$ can have up to $2n$
        distinct neighbours at level $(n+1)$.}
\end{itemize}

Because of these limiting factors, we cannot run through the census
and directly measure the shortest length or smallest excess height of any
simplification path from each node.
Instead we develop fast, localised algorithms that allow us to
bound these quantities from above.  To our delight,
these bounds turn out to be extremely effective in practice.
The details are as follows.

\subsection{Bounding excess heights} \label{s-analysis-height}

In this section we compute bounds $H_n$ so that, from every node at
level $n$ of the graph $\rpg{S^3}$, there is some simplification path of
excess height $\leq H_n$.  As in Theorem~\ref{t-results}, we compute these
bounds for each $n$ in the range $3 \leq n \leq 9$.

\begin{algorithm}[Algorithm for computing $H_n$] \label{a-height}
This algorithm runs by progressively building a subgraph $G \subset \rpg{S^3}$.
At all times we keep track of the number of distinct components of $G$
(which we denote by $c$) and the maximum level of any node in $G$
(which we denote by $\ell$).
\begin{enumerate}
    \item Initialise $G$ to all of level $n$ of $\rpg{S^3}$.
    This means that $G$ has no arcs, the number of components $c$ is
    just the number of nodes at level $n$, and the maximum level is
    $\ell = n$.

    \item While $c > 1$, expand the graph as follows:
    \begin{enumerate}[(a)]
        \item Construct all arcs from nodes in $G$ at level $\ell$
        to (possibly new) nodes in $\rpg{S^3}$ at level $\ell+1$.
        Insert these arcs and their endpoints into $G$.

        \item Update the number of components $c$, and increment $\ell$
        by one.
    \end{enumerate}

    \item Once we have $c=1$, output the final bound $H_n = \ell - n$
    and terminate.
\end{enumerate}
\end{algorithm}

In step~2(a) we construct arcs by performing 2-3 moves.
We only construct arcs from nodes \emph{already} in $G$,
which means we only work with a small portion of level $\ell$
for each $\ell > n$.
In step~2(b) we use union-find to update the number
of components in small time complexity.

It is clear that Algorithm~\ref{a-height} is correct for any $n \geq 3$:
once we have $c=1$ the subgraph $G$ is connected, which means there is a
path from any node at level $n$ to any other node at level $n$.
By Theorem~\ref{t-connected} at least one such node allows a 3-2 move,
and so any node at level $n$ has a
simplification path of excess height $\leq \ell$.

However, it is not clear that Algorithm~\ref{a-height} terminates:
it might be that \emph{every} simplification path from some
node at level $n$ passes through nodes that we never construct
at higher levels $\ell > n$.
Happily it does terminate for all
$3 \leq n \leq 9$, giving an output of $H_2 = 2$ each time.
Table~\ref{tab-height} shows how
the number of components $c$ changes throughout the algorithm in each case.

\begin{table}[htb]
\[ \small \begin{array}{l|r|r|r|r|r|r|r}
    \mbox{Input level $n$} & 3 & 4 & 5 & 6 & 7 & 8 & 9 \\
    \hline
    \mbox{Value of $c$ when $\ell = n$} &
    20 & 128 & 1\,297 & 13\,660 & 169\,077 & 2\,142\,197 & 28\,691\,150 \\
    \mbox{Value of $c$ when $\ell = n+1$} &
    8 & 50 & 196 & 1\,074 & 7\,784 & 64\,528 & 557\,428 \\
    \mbox{Value of $c$ when $\ell = n+2$} &
    1 & 1 & 1 & 1 & 1 & 1 & 1 \\
    \hline
    \mbox{Final bound $H_n$} &
    \mathbf{2} & \mathbf{2} & \mathbf{2} & \mathbf{2} & \mathbf{2} &
    \mathbf{2} & \mathbf{2}
\end{array} \]
\caption{Results obtained when running
    Algorithm~\ref{a-height} for $3 \leq n \leq 9$}
\label{tab-height}
\end{table}

It is straightforward to show that the space and time complexities of
Algorithm~\ref{a-height} are linear and log-linear respectively in the number
of nodes in $G$ (other small polynomial factors in $n$ and $\ell$
also appear).
Nevertheless, the memory requirements for $n=8$ were found to be extremely
large in practice ($\sim$29\,GB), and for $n=9$ they were too large for the
algorithm to run (estimated at 400--500\,GB).  In the case of $n=9$ a
\emph{two-phase} approach was necessary:
\begin{enumerate}
    \item Use Algorithm~\ref{a-height} for the transition
    from level $n$ to level $n+1$, and terminate if $H_n = 1$.
    \item From each node $v$ at level $n+1$, try all possible
    \emph{combinations} of a 2-3 move followed by a 3-2 move.
    Let $w$ be the endpoint of such a combination (so $w$ is also a
    node at level $n+1$).
    If $w \in G$ then merge the components and
    decrement $c$ if necessary.  Otherwise do nothing
    (since $w$ would never have been constructed in the original algorithm).
    \item If $c=1$ after this procedure then output $H_n=2$;
    otherwise terminate with no result.
\end{enumerate}

It is important to note that, if this two-phase approach \emph{does}
output a result, it will always be the same result as Algorithm~\ref{a-height}.
Essentially Step~2 simulates the transition from level $n+1$ to $n+2$ in
the original algorithm, with the advantage of a much smaller memory
footprint (since it does not store any nodes at level $n+2$), but
with the disadvantage that it cannot move on to level $n+3$ if required
(and so it cannot output any result if $H_n > 2$).

Of course by the time we reach $n=9$ there are reasons to suspect that
$H_n=2$ (following the pattern for $3 \leq n \leq 8$), and so this
two-phase method seems a reasonable (and ultimately successful) approach.
For $n=9$ the memory consumption was $\sim$50\,GB, which was (just)
within the capabilities of the host machine.

\subsection{Bounding path lengths}

Our next task is to compute bounds $L_n$ so that, from every node at
level $n$ of $\rpg{S^3}$, there is some simplification path of
length $\leq L_n$.  Once again we compute $L_n$ for $3 \leq n \leq 9$.

Because it is infeasible to perform arbitrary breadth-first searches through
$\rpg{S^3}$, we only consider paths that can be expressed as a series of
\emph{jumps}, where each jump involves a pair of 2-3 moves followed by a
pair of 3-2 moves.
This keeps the search space and memory usage small:
we always stay within levels $n$, $n+1$
and $n+2$, and we never need to explicitly store any nodes above
level $n$.  On the other hand, it means that our bounds $L_n$ are very
rough---there could be much shorter simplification paths that we do
not detect.

\begin{algorithm}[Algorithm for computing $L_n$] \label{a-length}
    First identify the set $I$ of all nodes at level $n$ of
    $\rpg{S^3}$ that have an arc running down to level $n-1$.
    Then conduct a breadth-first search across level $n$,
    beginning with the nodes in $I$ and using jumps as the
    steps in this breadth-first search.
    If $j$ is the maximum number of jumps required to reach
    any node in level $n$ from the initial set $I$,
    then output the final bound $L_n = 4j+1$.
\end{algorithm}

To identify the initial set $I$ we simply attempt to perform 3-2 moves.
When we process each node $v$, we must enumerate all jumps out from
$v$; that is, all combinations of two 2-3 moves followed by two 3-2 moves.
The number of such combinations is $O(n^4)$ in general.

This time we can guarantee both correctness and termination if
$3 \leq n \leq 9$.
Because $n \geq 3$ the initial set $I$ is non-empty
(Theorem~\ref{t-connected}), and from our height
experiments in Section~\ref{s-analysis-height} we know that our search
will eventually reach all of level $n$.  It follows that every node at
level $n$ of $\rpg{S^3}$ has a path of length $\leq 4j$ to some
$v \in I$, and therefore a simplification path of length $\leq 4j+1$.
Table~\ref{tab-length} shows how the search progresses for each $n$.

\begin{table}[htb]
\[ \small \begin{array}{l|r|r|r|r|r|r|r}
    \mbox{Input level $n$} & 3 & 4 & 5 & 6 & 7 & 8 & 9 \\
    \hline
    \mbox{Size of $I$} &
        3 & 46 & 504 & 6\,975 & 91\,283 & 1\,300\,709 & 18\,361\,866 \\
    \mbox{Nodes remaining} &
        17 & 82 & 793 & 6\,685 & 77\,794 & 841\,488 & 10\,329\,284 \\
    \mbox{Nodes remaining after 1 jump} &
        3 & 1 & 19 & 75 & 496 & 4\,222 & 31\,250 \\
    \mbox{Nodes remaining after 2 jumps} &
        0 & 0 & 1 & 1 & 0 & 6 & 12 \\
    \mbox{Nodes remaining after 3 jumps} & 0 & 0 & 0 & 0 & 0 & 0 & 0 \\
    \hline
    \mbox{Final bound $L_n$} &
    \mathbf{9} & \mathbf{9} & \mathbf{13} & \mathbf{13} & \mathbf{9} &
    \mathbf{13} & \mathbf{13}
\end{array} \]
\caption{Results obtained when running
    Algorithm~\ref{a-length} for $3 \leq n \leq 9$}
\label{tab-length}
\end{table}

This time the space and time complexities are linear and log-linear
respectively in the number of nodes at level $n$ (again with
further polynomial factors in $n$).  This is
considerably smaller than the number of nodes processed in
Algorithm~\ref{a-height}, and so for Algorithm~\ref{a-length} memory is
not a problem: the case $n=9$ runs in under 4\,GB.

\subsection{Parallelisation and performance} \label{s-analysis-perf}

For $n=9$, both Algorithms~\ref{a-height} and~\ref{a-length} have lengthy
running times: Algorithm~\ref{a-height} requires a very large number of
nodes to be processed at levels 9, 10 and 11 of the Pachner graph,
and Algorithm~\ref{a-length} spends significant time enumerating the
$O(n^4)$ available jumps from each node.

We can parallelise both algorithms by processing nodes simultaneously
(in step~2 of Algorithm~\ref{a-height}, and during each stage of the
breadth-first search in Algorithm~\ref{a-length}).  We must be careful
however to serialise any updates to the graph.

The experiments described here used
an 8-core 2.93\,GHz Intel Xeon X5570 CPU with 72\,GB of RAM
(using all cores in parallel).
With the serialisation bottlenecks, Algorithms~\ref{a-height} and
\ref{a-length} achieved roughly $90.5\%$ and $98.5\%$ CPU utilisation
for the largest case $n=9$, and ran for approximately
6 and 15 days respectively.
All code was written using the topological software package
{\regina} \cite{regina,burton04-regina}.

%
%

\section{Discussion} \label{s-conc}

As we have already noted, the bounds obtained in Section~\ref{s-analysis}
are astonishingly small.  Although we only consider $n \leq 9$,
this is not a small sample: the census includes $\sim 150$ million
triangulations including $\sim 31$ million one-vertex 3-spheres;
moreover, nine tetrahedra are enough to build complex and
interesting topological structures \cite{burton07-nor10,martelli01-or9}.
Our results lead us to the following conjectures:

\begin{conjecture} \label{cj-boundedheight}
    From any node at any level $n \geq 3$ of the graph $\rpg{S^3}$
    there is a simplification path of excess height $\leq 2$.
\end{conjecture}

If true, this result (combined with Theorem~\ref{t-numvert})
would reduce Mijatovi{\'c}'s bound in Theorem~\ref{t-mij}
from $\exp(O(n^2))$ to $\exp(O(n\log n))$ for one-vertex triangulations
of the 3-sphere.  Furthermore, it would help explain why 3-sphere
triangulations are so easy to simplify in practice.

There are reasons to believe that a proof might be possible.
As a starting point,
a simple Euler characteristic argument shows that
every closed 3-manifold triangulation has an edge of degree $\leq 5$;
using \emph{at most two} ``nearby'' 2-3 moves, this edge can be made
degree~three (the setting for a possible 3-2 simplification).
The details will appear in the full version of this paper.

\begin{conjecture} \label{cj-boundedmoves}
    From any node at any level $n \geq 3$ of the graph $\rpg{S^3}$
    there is a simplification path of length $\leq 13$.
\end{conjecture}

This is a bolder conjecture, since the length experiments are less
consistent in their results.  However, the fact remains that every
3-sphere triangulation of size $n \leq 9$ can be simplified after
just three jumps, and this number does not rise between $n=5$
and $n=9$.

If true, this second conjecture would yield an immediate polynomial-time
3-sphere recognition algorithm: for any triangulation of size
$n \geq 3$ we can enumerate all $O(n^{4 \times 3})$ combinations of three
jumps, and test each resulting triangulation for a 3-2 move down to
$n-1$ tetrahedra.  By repeating this process $n-2$ times, we will
achieve either a recognisable 2-tetrahedron triangulation of the 3-sphere,
or else a proof that our input is not a 3-sphere triangulation.

Even if Conjecture~\ref{cj-boundedmoves} is false and the length
bounds do grow with $n$, this growth rate appears to be extremely slow.
A growth rate of $L_n \in O(\log n)$ or even $O(\sqrt{n})$
would still yield the first known sub-exponential 3-sphere recognition
algorithm (using the same procedure as above), which would be a
significant theoretical breakthrough in algorithmic 3-manifold topology.

Looking forward, it is natural to ask whether this behaviour
extends beyond the 3-sphere to triangulations of
arbitrary 3-manifolds.  Initial experiments suggest ``partially'':
the Pachner graphs of other 3-manifolds also appear to be
remarkably well-connected, though not enough to support results as
strong as Conjectures~\ref{cj-boundedheight} and~\ref{cj-boundedmoves} above.
We explore these issues further in the full version of this paper.

%
%

\section*{Acknowledgements}

The author is grateful to the Australian Research Council for their
support under the Discovery Projects funding scheme (project DP1094516).
Computational resources used in this work
were provided by the Queensland Cyber Infrastructure Foundation
and the Victorian Partnership for Advanced Computing.

%
%

\small
\bibliographystyle{amsplain}
\bibliography{pure}

%
%

\bigskip
\smallskip
\noindent
Benjamin A.~Burton \\
School of Mathematics and Physics, The University of Queensland \\
Brisbane QLD 4072, Australia \\
(bab@maths.uq.edu.au)

%
%

\normalsize
\appendix

\section*{Appendix: Additional proofs}

Here we offer full proofs for Theorem~\ref{t-numvert} and
Lemma~\ref{l-can-fast}, which were omitted from the main text to
simplify the exposition.

\setcounter{theorem}{\arabic{ctr-numvert}}
\begin{theorem}
    The number of distinct isomorphism classes of 3-manifold
    triangulations of size $n$ grows at an asymptotic rate of
    $\exp(\Theta(n\log n))$.
\end{theorem}

\begin{proof}
    An upper bound of $\exp(O(n\log n))$ is easy to obtain.
    If we count all possible gluings of tetrahedron faces, without regard
    for isomorphism classes or other constraints (such as the need for
    the triangulation to represent a closed 3-manifold), we obtain an
    upper bound of
    \[ \left[(4n-1)\times(4n-3)\times\cdots\times3\times1 \right]
        \cdot 6^{2n} < (4n)^{2n} \cdot 6^{2n} \in \exp(O(n\log n)). \]

    Proving a lower bound of $\exp(\Omega(n\log n))$ is more
    difficult---the main complication is that most pairwise
    identifications of tetrahedron faces do not yield a 3-manifold at all
    \cite{dunfield06-random-covers}.  We work around this by first counting
    \emph{2-manifold} triangulations (which are much
    easier to obtain), and then giving a construction that ``fattens''
    these into 3-manifold triangulations without
    introducing any unwanted isomorphisms.

    To create a 2-manifold triangulation of size $2m$ (the size
    must always be even), we identify the $6m$ edges of $2m$
    distinct triangles in pairs.  Any such identification will always
    yield a closed 2-manifold (that is, nothing can ``go wrong'',
    in contrast to the three-dimensional case).

    There is, however, the issue of connectedness to deal with
    (recall from the beginning of Section~\ref{s-prelim} that all
    triangulations in this paper are assumed to be connected).  To ensure that
    a labelled 2-manifold triangulation is connected, we insist that for each
    $k=2,3,\ldots,2m$, the first edge of the triangle labelled $k$ is
    identified with
    some edge from one of the triangles labelled $1,2,\ldots,k-1$.  Of course
    many connected labelled 2-manifold triangulations do not have this
    property, but since we are proving a lower bound this does not matter.

    We can now place a lower bound on the number of labelled
    2-manifold triangulations.  First we choose
    which edges to pair with the first edges from triangles
    $2,3,\ldots,2m$; from the property above we have
    $3 \times 4 \times \ldots \times 2m \times (2m+1) = \frac12 (2m+1)!$
    choices.
    We then pair off the remaining $2m+2$ edges, with
    $(2m+1) \times (2m-1) \times \ldots \times 3 \times 1 =
    (2m+1)!/2^m m!$ possibilities overall.
    Finally we note that each of the $3m$ pairs of edges can be identified
    using one of two possible orientations.
    The total number of labelled 2-manifold triangulations is therefore
    at least
    \[ \frac{(2m+1)!}{2} \cdot \frac{(2m+1)!}{2^m m!} \cdot 2^{3m}
        = \frac{(2m+1)! \cdot (2m+1)! \cdot 2^{2m}}{2 \cdot m!}. \]

    Each isomorphism class can contain at most
    $(2m)! \cdot 6^{2m}$ labelled triangulations, and so the number of
    distinct \emph{isomorphism classes} of 2-manifold triangulations is
    bounded below by
    \begin{align*}
    \frac{(2m+1)! \cdot (2m+1)! \cdot 2^{2m}}
            {2 \cdot m! \cdot (2m)! \cdot 6^{2m}} &=
        \frac{(2m+1) \cdot (2m+1)!}{2 \cdot m! \cdot 3^{2m}} \\
        &> (2m+1) \times 2m \times \cdots \times (m+2) \times (m+1) \times
            \left(\tfrac{1}{9}\right)^m \\
        &> (m+1)^{m+1} \cdot \left(\tfrac{1}{9}\right)^m \\
        &\in \exp(\Omega(m\log m)).
    \end{align*}

    We fatten each 2-manifold triangulation into a 3-manifold triangulation
    as follows.  Let $F$ denote the closed 2-manifold described by the
    original triangulation.
    \begin{enumerate}
        \item Replace each triangle with a prism and glue the vertical
        faces of adjacent prisms together, as illustrated in
        Figure~\ref{sub-fatten-prisms}.
        This represents a \emph{bounded} 3-manifold, which is the
        product space $F \times I$.

        \item Cap each prism at both ends with a triangular pillow,
        as illustrated in Figure~\ref{sub-fatten-pillow}.
        The two faces of each pillow are glued to the top and bottom of
        the corresponding prism, effectively converting each prism into
        a solid torus.  This produces the \emph{closed} 3-manifold
        $F \times S^1$, and the complete construction is illustrated in
        Figure~\ref{sub-fatten-all}.

        \item Triangulate each pillow using two tetrahedra, which are joined
        along three internal faces surrounding an internal vertex.
        Triangulate each prism using $14$ tetrahedra, which again all
        meet at an internal vertex.
        Both triangulations are illustrated in Figure~\ref{sub-fatten-tri}.
    \end{enumerate}

    \begin{figure}[htb]
        \centering
        \begin{tabular}{c@{\qquad\qquad}c}
        \subfigure[Replacing triangles with prisms]{%
            \label{sub-fatten-prisms}%
            \includegraphics[scale=0.45]{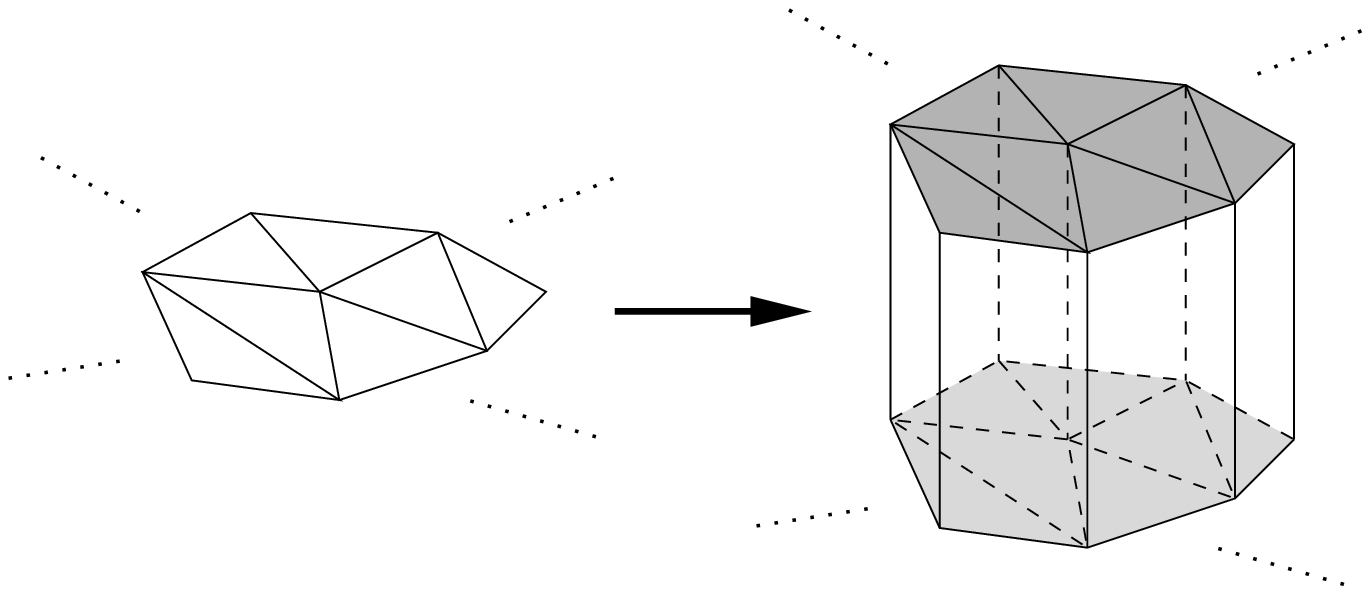}}%
        &
        \subfigure[Capping prisms with pillows]{%
            \hspace{2cm}%
            \label{sub-fatten-pillow}%
            \includegraphics[scale=0.45]{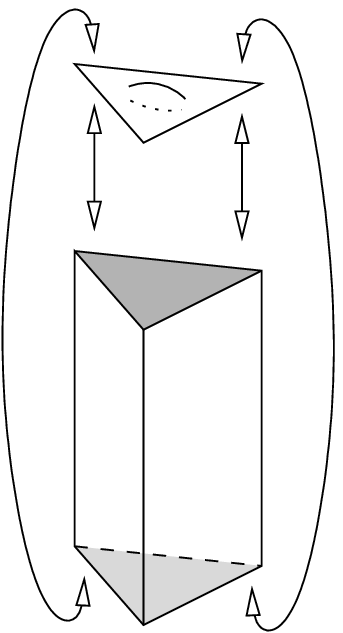}%
            \hspace{2cm}}
        \\
        \subfigure[The complete construction]{%
            \label{sub-fatten-all}%
            \includegraphics[scale=0.45]{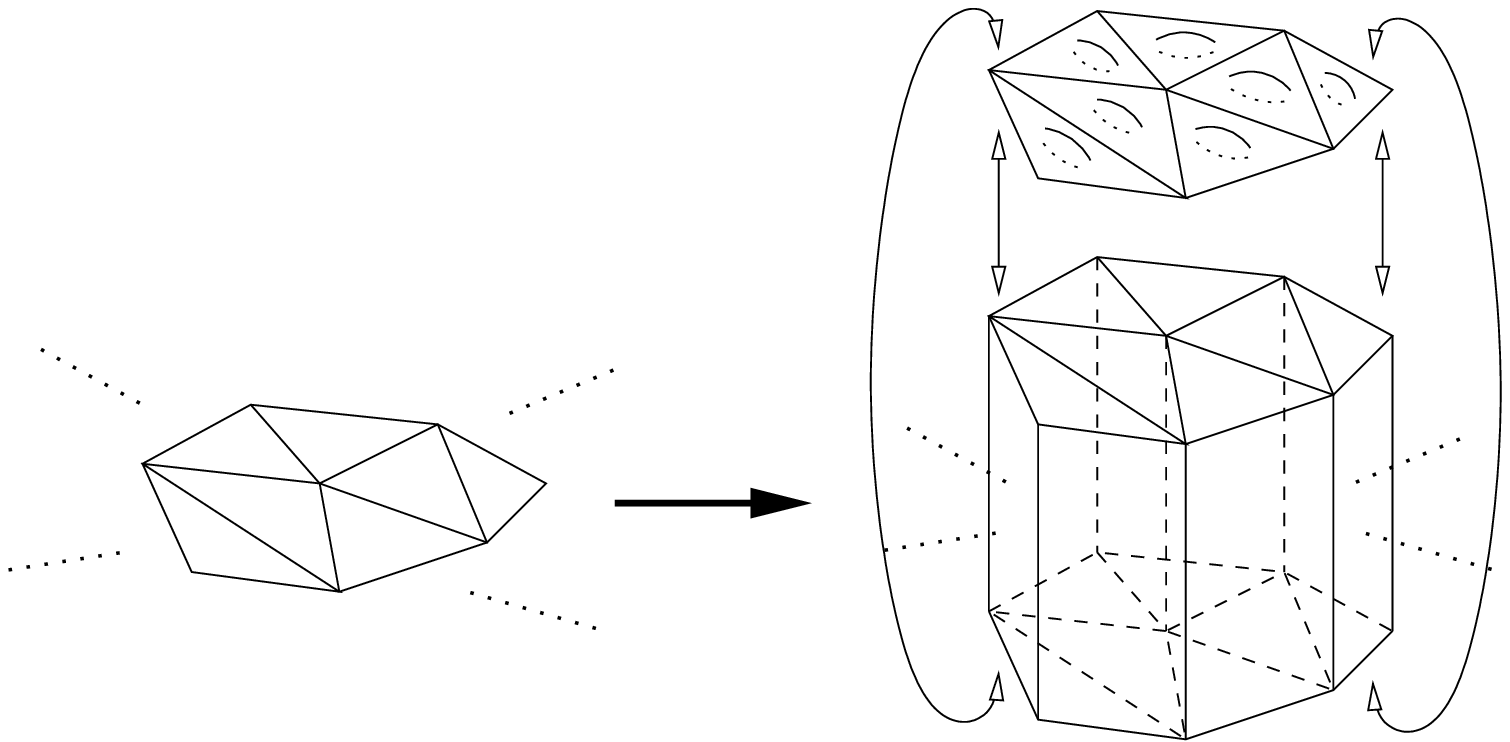}}
        &
        \subfigure[Triangulating prisms and pillows]{%
            \hspace{2cm}%
            \label{sub-fatten-tri}%
            \includegraphics[scale=0.45]{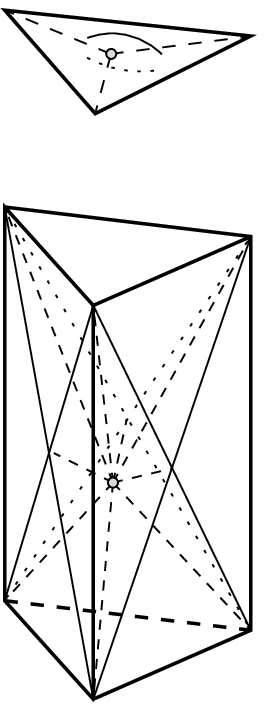}%
            \hspace{2cm}}
        \end{tabular}
        \caption{Fattening a 2-manifold triangulation into a 3-manifold
            triangulation}
        \label{fig-fatten}
    \end{figure}

    If the original 2-manifold triangulation uses $2m$ triangles, the
    resulting 3-manifold triangulation uses $n=32m$ tetrahedra.
    Moreover, if two 3-manifold triangulations obtained using this
    construction are isomorphic, the original 2-manifold triangulations
    must also be isomorphic.  The reason for this is as follows:
    \begin{itemize}
        \item Any isomorphism between two such 3-manifold triangulations
        must map triangular pillows to triangular pillows.  This is
        because the internal vertex of each triangular pillow meets only
        two tetrahedra, and no other vertices under our construction
        have this property.

        \item By ``flattening'' the triangular pillows into
        2-dimensional triangles, we thereby obtain an
        isomorphism between the underlying 2-manifold triangulations.
    \end{itemize}

    It follows that, for $n=32m$, we obtain a family of
    $\exp(\Omega(m\log m)) = \exp(\Omega(n\log n))$ pairwise
    non-isomorphic 3-manifold triangulations.

    This result is easily extended to $n \not\equiv 0 \bmod 32$.
    Let $V_n$ denote the number of distinct isomorphism classes of
    3-manifold triangulations of size $n$.
    \begin{itemize}
        \item Each triangulation of size $n$ has at least $n-1$
        distinct 2-3 moves available (since any face joining two
        distinct tetrahedra defines a 2-3 move, and there are at least
        $n-1$ such faces).

        \item On the other hand, each triangulation of size $n+1$
        has at most $6(n+1)$ distinct 3-2 moves available (since each 3-2
        move is defined by an edge that meets three distinct tetrahedra,
        and the triangulation has at most $6(n+1)$ edges in total).
    \end{itemize}

    It follows that $V_{n+1} \geq V_n \cdot \frac{n-1}{6(n+1)} \geq V_n/18$
    for any $n > 1$.  This gives
    $V_{32m+k} \geq V_{32m} / 18^{31}$ for sufficiently large $m$ and
    all $0 \leq k < 32$, and so we obtain
    $V_n \in \exp(\Omega(n\log n))$ with no restrictions on $n$.
\end{proof}

\begin{remark}
    Of course, we expect that $V_{n+1} \gg V_n$ (and indeed we see this
    in the census).  The bounds
    that we use to show $V_{n+1} \geq V_n/18$ in the proof above are
    very loose, but they are sufficient for the asymptotic result that we seek.
\end{remark}

\setcounter{theorem}{\arabic{ctr-can-fast}}
\begin{lemma}
    For any triangulation $\tri$ of size $n$, there are precisely
    $24n$ canonical labellings of $\tri$, and these can be enumerated in
    $O(n^2\log n)$ time.
\end{lemma}

\begin{proof}
    For $n=1$ the result is trivial, since all $24=4!$ possible labellings
    are canonical.  For $n>1$ we observe that, if we choose
    (i)~any one of the $n$ tetrahedra to label as tetrahedron~1, and
    (ii)~any one of the $24$ possible labellings of its four vertices,
    then there is one and only one way to extend these choices to a canonical
    labelling of $\tri$.

    To see this, we can walk through the list of faces
    $F_{1,1},F_{1,2},F_{1,3},F_{1,4},F_{2,1},\ldots,F_{n,4}$,
    where $F_{t,i}$ represents face $i$ of tetrahedron $t$.
    The first face amongst $F_{1,1},\ldots,F_{1,4}$ that is joined to an
    unlabelled tetrahedron must in fact be joined to tetrahedron~2 using
    the identity map; this allows us to deduce tetrahedron~2 as well as
    the labels of its four vertices.

    We inductively extend the labelling in this manner: once we have
    labelled tetrahedra $1,\ldots,k$ and their corresponding vertices,
    the first face amongst $F_{1,1},\ldots,F_{k,4}$ that is joined to
    an unlabelled tetrahedron must give us tetrahedron $k+1$ and
    the labels for its four vertices (again using the identity map).
    The resulting labelling is canonical, and all of the labels can be
    deduced in $O(n\log n)$ time using a single pass through the list
    $F_{1,1},\ldots,F_{n,4}$.  The $\log n$ factor is required for
    manipulating tetrahedron labels, each of which requires $O(\log n)$ bits.

    It follows that there are precisely $24n$ canonical labellings
    of $\tri$, and that these can be enumerated in $O(n^2\log n)$ time
    using $24n$ iterations of the procedure described above.
\end{proof}

\end{document}